\documentclass[11pt,fleqn]{article}

\usepackage[paper=a4paper]
 {geometry}

\usepackage{paragraphs}
\usepackage{hyperref}
\usepackage{amsthm,thmtools}

\usepackage[utf8]{inputenc}
\usepackage[spanish,english]{babel}
\usepackage{enumitem}
\usepackage[osf,noBBpl]{mathpazo}
\usepackage[alphabetic,initials]{amsrefs}
\usepackage{amsfonts,amssymb,amsmath}
\usepackage{mathtools}
\usepackage{graphicx}
\usepackage[poly,arrow,curve,matrix]{xy}
\usepackage{wrapfig}
\usepackage{xcolor}
\usepackage{helvet}
\usepackage{stmaryrd}

\declaretheoremstyle[headformat=swapnumber, spaceabove=\paraskip,
bodyfont=\itshape]{mystyle}
\declaretheoremstyle[headformat=swapnumber, spaceabove=\paraskip,
bodyfont=\normalfont]{mystyle-plain}
\declaretheorem[name=Lemma, sibling=para, style=mystyle]{Lemma}

\declaretheoremstyle[numbered=no, spaceabove=\paraskip,
bodyfont=\itshape]{mystyle-empty}
\declaretheoremstyle[numbered=no, spaceabove=\paraskip,
bodyfont=\itshape]{mystyle-empty-plain}
\declaretheorem[name=Lemma, style=mystyle-empty]{Lemma*}
\declaretheorem[name=Proposition, style=mystyle-empty]{Proposition*}
\declaretheorem[name=Theorem, style=mystyle-empty]{Theorem*}
\declaretheorem[name=Corollary, style=mystyle-empty]{Corollary*}
\declaretheorem[name=Definition, style=mystyle-empty]{Definition*}
\declaretheorem[name=Remark, style=mystyle-empty]{Remark*}

\declaretheoremstyle[
    headformat={{\bfseries\NUMBER.}{\itshape\NAME}\NOTE\ignorespaces},
    spaceabove=\paraskip, 
    headpunct={.},
    headfont=\itshape,
    bodyfont=\normalfont
    ]{mystyle-plain}

\makeatletter
\renewenvironment{proof}[1][\textit{Proof}]{\par
 \pushQED{\qed}%
 \normalfont \topsep.75\paraskip\relax
 \trivlist
 \item[\hskip\labelsep
    \itshape
  #1\@addpunct{.}]\ignorespaces
}{%
 \popQED\endtrivlist\@endpefalse
}
\makeatother

\setenumerate[0]{label=(\alph*)}
\newcommand\NN{\mathbb N}

\newcommand\ZZ{\mathbb Z}

\newcommand\ot{\otimes}

\renewcommand\to{\longrightarrow}
\renewcommand\phi{\varphi}
\renewcommand\k{\Bbbk}

\renewcommand\O{\mathcal O}
\newcommand\R{\mathcal R}
\newcommand\op{\mathsf{op}}
\newcommand\K{\mathcal K}
\newcommand\D{\mathcal D}
\newcommand\A{\mathcal A}

\DeclareMathOperator\Mod{\mathsf{Mod}}
\DeclareMathOperator\Gr{\mathsf{Gr}}
\DeclareMathOperator\Hom{\mathsf{Hom}}

\DeclareMathOperator\HOM{\underline{\mathsf{Hom}}}

\DeclareMathOperator\injdim{injdim}
\DeclareMathOperator\projdim{pdim}

\DeclareMathOperator\supp{supp}
\DeclareMathOperator\Id{Id}

\DeclareMathOperator\Res{\mathsf{Res}}
\DeclareMathOperator\coker{coker}
\DeclareMathOperator\nat{nat}

\title{
Change of grading, injective dimension and dualizing complexes
}

\author{A. Solotar,
P. Zadunaisky\footnote{This work has been supported by the projects 
UBACYT 20020130 100533BA, pip-conicet 11220150100483CO, and 
MATHAMSUD-REPHOMOL. The first named author is a research member of CONICET 
(Argentina). The second named author is a FAPESP PostDoc Fellow, grant: 
2016-25984-1 S\~ao Paulo Research Foundation (FAPESP).}
}
\date{}

\begin{document}
\maketitle

\begin{abstract}
Let $G,H$ be groups, $\phi: G \to H$ a group morphism, and $A$ a $G$-graded 
algebra. The morphism $\phi$ induces an $H$-grading on $A$, and on any 
$G$-graded $A$-module, which thus becomes an $H$-graded $A$-module.
Given an injective $G$-graded $A$-module, we give bounds for its injective
dimension when seen as $H$-graded $A$-module. Following ideas by Van den
Bergh, we give an application of our results to the stability of dualizing
complexes through change of grading.
\end{abstract}

\textbf{2010 MSC:} 16D50, 16E10, 16E65, 16W50, 18G05.

\textbf{Keywords:} injective modules, change of grading, dualizing complexes.

\section{Introduction}
Graded rings are ubiquitous in algebra. One of the main reasons is that the
presence of a grading simplifies proofs and allows to generalize many results
(for example, the theories of commutative and noncommutative graded algebras 
are easier to reconcile than their ungraded counterparts). Furthermore, 
results can often be transfered from the graded to the ungraded context 
through standard techniques. In more categorical terms, there is a natural 
forgetful functor from the category $\Gr_G A$ of graded modules over a 
$G$-graded algebra $A$, to the category $\Mod A$ of modules over $A$, and the 
challenge is to find a way to transfer information in the opposite direction. 
When $G = \ZZ$ this is usually done through ``filtered-and-graded'' arguments 
and spectral sequences. In this article we exploit a different technique, 
namely the existence of three functors $\phi_!, \phi^*, \phi_*$, where 
$\phi_!: \Gr_G A \to \Mod A$ is the usual forgetful functor (sometimes also 
called the \emph{push-down} functor), $\phi^*$ is its right adjoint, and 
$\phi_*$ is the right adjoint of $\phi^*$. 
This technique has two advantages over the usual filtered-and-graded methods, 
namely that it does not depend on the choice of a non-canonical filtration, 
and that the group $G$ is arbitrary. Its main drawback is that the functors in 
this triple do not preserve finite generation, noetherianity, or other 
``finiteness'' properties unless further hypotheses are in place.

The problem we consider is the following. Suppose you are given an injective 
object $I$ in the category $\Gr_\ZZ A$. In general $I$ is not 
injective as $A$-module, but if $A$ is noetherian then 
its injective dimension is at most one. Now, what happens if we consider 
gradings by more general groups? In general, given groups $G, H$ and 
a group morphism $\phi: G \to H$, any $G$-graded object can be seen as an 
$H$-graded object through $\phi$, see paragraph \ref{cog-functors}. In 
particular a $G$-graded algebra $A$ inherits an $H$-grading, and there is a 
natural functor $\phi_!: \Gr_G A \to \Gr_H A$, between the categories of 
$G$-graded and $H$-graded modules. The question thus becomes: given an 
injective object $I$ in $\Gr_G A$, what is the injective dimension of 
$\phi_!(I)$ in $\Gr_H A$?

This question has been considered several times in the literature, but it has 
received no unified treatment. A classical result of R. Fossum and H.-B. Foxby 
\cite{FF-graded}*{Theorem 4.10} states that if $A$ is $\ZZ$-graded noetherian 
and commutative then a $\ZZ$-graded-injective module has injective dimension 
at most $1$. M. Van den Bergh claims in the article 
\cite{VdB-existence-dc}*{below Definition 6.1} that this result extends to the 
noncommutative case if the algebra is $\NN$-graded and $A_0$ is equal to the 
base field; a proof of this fact can be found in the preprint \cite{Yek-note}. 
Other antecedents include \cite{Eks-auslander}, where it is shown that if $A$ 
is a noetherian $\ZZ$-graded algebra then the injective dimension of $A$ is 
finite if and only if its graded injective dimension is finite. Following the 
ideas of \cite{Lev-ncreg}*{section 3}, one can show that if $A$ is 
$\NN$-graded and noetherian, and $M$ is a $\ZZ$-graded module such that $M_n 
= 0$ for $n \ll 0$, then the graded injective dimension of $M$ coincides with 
its injective dimension as $A$-module. Most of these results are obtained by 
the usual route of going from ungraded to graded objects through filtrations 
and spectral sequences. The only result that we could find in the literature 
regarding injective modules graded by groups other than $\ZZ$ states that if 
$A$ is graded over a finite group then a graded module is graded injective if 
and only if it is injective \cite{NV-graded-book3}*{2.5.2}. 

In order to give a general answer to the question we work with the functors 
$\phi_!, \phi^*, \phi_*$ mentioned above, which were originally introduced by 
A. Polishchuk and L. Positselski in \cite{PP-secondHH}. These functors, 
collectively called the \emph{change of 
grading functors}, turn out to be particularly well-adapted to the transfer of 
information of homological nature. Our main result, which includes most of the 
previous ones as special cases, is the following.
\begin{Theorem*}
Let $\phi: G \to H$ be a group morphism, let $L = \ker \phi$ and let $d$ be 
the projective dimension of the trivial $L$-module $\k$. Let $A$ be a 
$G$-graded noetherian algebra, and let $I$ be an injective object of $\Gr_G 
A$. Then the injective dimension of $\phi_!(I)$ is at most $d$. 
\end{Theorem*}
The proof depends on two facts. First, that if $I$ is 
$G$-graded injective then $\phi_!(I)$ is an injective object in the additive 
subcategory generated by all modules of the form $\phi_!(M)$ with $M$ a 
$G$-graded $A$-module; in other words, modules in the image of $\phi_!$ are 
$\Hom_A^H(-,\phi_!(I))$-acyclic and hence can be used to build acyclic 
resolutions, see Lemma \ref{L:acyclic}. The second is a result of independent 
interest, stating that given an $H$-graded $A$-module $N$ we can obtain a 
resolution of $N$ by objects in the additive category generated by 
$\phi_!(\phi^*(N))$, see Proposition \ref{P:resolution}; this resolution 
can be used to calculate the $H$-graded extension modules between $N$ and 
$\phi_!(I)$, which gives the desired bound. 

\bigskip
The article is structured as follows. In Section \ref{COG-FUNCTORS} we review 
some basic facts on the category of graded modules and recall some general 
properties of the change of grading functors established in the article 
\cite{RZ-twisted}. In Section \ref{COG-INJDIM} we 
prove our main results on how regrading affects injective dimension. Finally 
in Section \ref{COG-DC} we give similar results at the derived level and use 
them to study the behavior of dualizing complexes with respect to regradings,
a question originally raised by Van den Bergh in \cite{VdB-existence-dc}.

\bigskip
Throughout the article $\k$ is a commutative ring, and unadorned $\hom$ 
spaces and tensor products are always over $\k$. Also all modules over rings
are left modules unless otherwise stated. The letters $G, H$ will always denote
groups, and $\phi: G \to H$ will be a group morphism.

\bigskip
\textbf{Acknowledgements:} The authors would like to thank Mariano 
Suárez-Álvarez for a careful reading of a previous version of this article.

\section{The change of grading functors}
\label{COG-FUNCTORS}

\paragraph
\label{G-graded-vector-spaces}
A $G$-graded $\k$-module is a $\k$-module $V$ with a fixed 
decomposition $V = \bigoplus_{g \in G} V_g$; we say that $v \in V$ is 
homogeneous of degree $g$ if $v \in V_g$, and $V_g$ is called the 
$g$-homogeneous component of $V$. We usually say graded instead of $G$-graded 
if $G$ is clear from the context.

Given two $G$-graded modules $V$ and $W$, their tensor product is also a 
$G$-graded module, where for each $g \in G$ 
\begin{align*}
(V \ot W)_g = \bigoplus_{g' \in G} V_ {g'} \ot W_{(g')^{-1}g}
\end{align*}
A map between graded $\k$-modules $f: V \to W$ is said to be 
\emph{$G$-homogeneous}, or simply homogeneous, if $f(V_g) \subset W_g$ for all 
$g \in G$. By definition, a homogeneous map $f: V \to W$ induces maps 
$f_g: V_g \to W_g$ for each $g \in G$, and $f = \bigoplus_{g \in G} f_g$; 
we refer to $f_g$ as the homogeneous component of degree $g$ of $f$. The 
\emph{support} of a $G$-graded $\k$-module $V$ is $\supp V = \{g \in G \mid 
V_g \neq 0\}$.

The category $\Gr_G \k$ has $G$-graded modules as objects and homogeneous 
$\k$-linear maps as morphisms. Kernels and cokernels of homogeneous
maps between graded $\k$-modules are graded in a natural way, so a complex
\[
 0 \to V' \to V \to V'' \to 0
\]
in $\Gr_G \k$ is a short exact sequence if and only if it is a short
exact sequence of $\k$-modules, or equivalently if for each $g \in G$
the sequence formed by taking $g$-homogeneous components is exact.

Given an object $V$ in $\Gr_G \k$ and $g \in G$, we denote by $V[g]$ the
$G$-graded $\k$-module whose homogeneous component of degree $g'$ is
$V[g]_{g'} = V_{g'g}$. This gives a natural autoequivalence of $\Gr_G \k$.

\paragraph
\label{G-graded-algebras}
We now recall the general definitions regarding $G$-graded $\k$-algebras.
The reader is referred to \cite{NV-graded-book3}*{Chapter 2} for proofs and 
details.

A $G$-graded $\k$-algebra is a $G$-graded $\k$-module $A$ which is also a 
$\k$-algebra, such that for all $g,g' \in G$ and all $a \in A_g, a' \in A_{g'}$
we have $aa' \in A_{gg'}$. If $A$ is a $G$-graded algebra then its
\emph{structural map} $\rho: A \to A \ot \k[G]$ is defined as $a \in A_g 
\mapsto a \ot g \in A_g \ot \k[G]_g$ for each $g \in G$; the fact that $A$ is 
a $G$-graded algebra implies that this is a morphism of algebras.

A $G$-graded $A$-module is an $A$-module $M$ which is also a $G$-graded 
$\k$-module such that for each $g,g' \in G$ and all $a \in A_g, 
m \in M_{g'}$ it happens that $am \in M_{gg'}$. Once again, we usually say 
graded instead of $G$-graded. We say that $A$ is graded left noetherian if
every graded $A$-submodule of a finitely generated graded $A$-module is
also finitely generated. If $G$ is a polycyclic-by-finite group then $A$ is
graded noetherian if and only if it is noetherian 
\cite{CQ-polycyclic}*{Theorem 2.2}.

We denote by $\Gr_G A$ the category whose objects are $G$-graded
$A$-modules and whose morphisms are $G$-homogeneous $A$-linear maps. 
Notice that if $M$ is a graded $A$-module then the graded $\k$-module $M[g]$
is also a graded $A$-module, with the same underlying $A$-module structure,
so shifting also induces an autoequivalence of $\Gr_G A$. 

The category $\Gr_G A$ has arbitrary direct sums and products. The direct sum
of graded modules is again graded in an obvious way, but this is not the case
for direct products. Given a collection of graded $A$-modules $\{V^i \mid i 
\in I\}$, their direct product is the graded $A$-module whose homogeneous 
decomposition is given by
\begin{align*}
\bigoplus_{g \in G} \prod_{i \in I} V^i_g.
\end{align*}
In other words, the forgetful functor $\O: \Gr_G A \to \Mod A$ preserves
direct sums, but not direct products.

The category $\Gr_G A$ is a Grothendieck category with enough projective and 
injective objects. Given an object $M$ of $\Gr_G A$, we will denote by 
$\projdim_A^G M$ and $\injdim_A^G M$ its projective and injective dimensions, 
respectively. Given two graded $A$-modules $M, N$ we denote by $\Hom^G_A(M,N)$ 
the $\k$-module of all $G$-homogeneous $A$-linear morphisms from $M$ to $N$. 
Since $\Gr_G A$ has enough injectives, we can define for each $i \geq 0$ the 
$i$-th right derived functor of $\Hom^G_A$, which we denote by $\R^i\Hom^G_A$. 

There is also an enriched homomorphism functor $\HOM_A^G$, given by
\begin{align*}
\HOM_A^G(M,N) = \bigoplus_{g \in G} \Hom_A^G(M,N[g]),
\end{align*}
which is a $G$-graded $\k$-submodule of $\Hom_\k(M,N)$. We denote its right
derived functors by $\R^i\HOM_A^G$.

\paragraph
\label{cog-functors}
Let $A$ be a $G$-graded $\k$-algebra. As shown in 
\cite{RZ-twisted}*{Section 1.3}, a group homomorphism $\phi: G \to H$
induces functors $\phi_!, \phi_*: \Gr_G A \to \Gr_H A$ and $\phi^*: 
\Gr_H A \to \Gr_G A$. We quickly review the construction for completeness.

Let $V$ be a $G$-graded $\k$-module. We define $\phi_!(V)$ to be the 
$H$-graded $\k$-module whose homogeneous component of degree $h \in H$ is 
given by
\begin{align*}
\phi_!(V)_h
 &= \bigoplus_{\{g \in G \mid \phi(g) = h\}} V_g.
\end{align*}
Analogously given a map $f: V \to W$ between $G$-graded $\k$-modules, we
define $\phi_!(f)$ to be the $\k$-linear map whose homogeneous component of 
degree $h \in H$ is given by
\begin{align*}
\phi_!(f)_h
 &= \bigoplus_{\{g \in G \mid \phi(f) = h\}} f_g.
\end{align*}
Notice that $\phi_!(V)$ has the same underlying $\k$-module as $V$. In 
particular, $\phi_!(A)$ is an $H$-graded $\k$-algebra which is equal to $A$ as 
$\k$-algebra, and if $V$ is a $G$-graded $A$-module then $\phi_!(V)$ is an 
$H$-graded $\phi_!(A)$-module with the same underlying $A$-module structure. 
Since the action of $A$ remains unchanged, if $f$ is $A$-linear then so is 
$\phi_!(f)$. This defines the functor $\phi_!: \Gr_G A \to \Gr_H \phi_!(A)$. 
From now on we usually write $A$ instead of $\phi_!(A)$ to lighten up the 
notation, since the context will make it clear whether we are considering it 
as a $G$-graded or as an $H$-graded algebra.

We define $\phi_*(V)$ and $\phi_*(f)$, to be the $H$-graded $\k$-module,
and $H$-homogeneous map whose homogeneous components of degree $h \in H$ 
are given by
\begin{align*}
\phi_*(V)_h
 &= \prod_{\{g \in G \mid \phi(g) = h\}} V_g,
&\phi_*(f)_h
 &= \prod_{\{g \in G \mid \phi(f) = h\}} f_g,
\end{align*}
respectively. If $V$ is also an $A$-module, we define the action of a 
homogeneous element $a \in A_{g'}$ with $g' \in G$ over an element 
$(v_g)_{g \in \phi^{-1}(h)} \in \phi_*(V)_h$ as $a(v_g) = (av_g)$. With this 
action $\phi_*(V)$ becomes an $H$-graded $A$-module, and we have defined the 
functor $\phi_*: \Gr_G A \to \Gr_H A$.

Now let $V',W'$ be $H$-graded $\k$-modules and let $f': V' \to W'$ be a 
homogeneous map. We set $\phi^*(V') \subset V' \ot \k[G]$ to be the 
subspace generated by all elements of the form $v \ot g$ with $v \in V'$ 
homogeneous of degree $\phi(g)$, and $\phi^*(f)(v \ot g) = f(v) \ot g$. In 
other words, for each $g \in G$ the homogeneous components of $\phi^*(V')$
and $\phi(f')$ of degree $g$ are given by
\begin{align*}
\phi^*(V')_g 
 &= V'_{\phi(g)} \ot \k g,
 &f_g
 &= f_{\phi(g)} \ot \Id.
\end{align*}
If $V'$ is an $H$-graded $A$-module, then $V' \ot \k[G]$ is an $A \ot 
\k[G]$-module, and it is an induced $A$-module through the structure map 
$\rho: A \to A \ot \k[G]$; it is immediate to check that with this action it 
becomes a $G$-graded $A$-module with $(V' \ot \k[G])_g = V' \ot \k g$ for 
each $g \in G$, and that $\phi^*(V') \subset V' \ot \k[G]$ is a $G$-graded 
$A$-submodule. It is also easy to check that if $f'$ is homogeneous and 
$A$-linear then so is $\phi^*(f')$. Thus we have defined a functor $\phi^*: 
\Gr_H A \to \Gr_G A$.

\paragraph
\label{P:adjoint}
We refer to $\phi_!, \phi^*$ and $\phi_*$ collectively as the \emph{change of
grading functors}. It is clear from the definitions that the change of grading 
functors are exact, and that $\phi_!, \phi_*$ reflect exactness, i.e. a complex
is exact if and only if its image by any of them is also exact. The functor 
$\phi^*$ reflects exactness if and only if $\phi$ is surjective. As mentioned 
before, we have some adjointness relations between these functors. 
\begin{Proposition*}[\cite{RZ-twisted}*{Proposition 3.2.1}]
The functor $\phi^*$ is right adjoint to $\phi_!$ and left adjoint to $\phi_*$.
\end{Proposition*}
\begin{proof}
Let $M$ be an object of $\Gr_G A$ and $N$ an object of $\Gr_H A$. We define 
maps
\begin{align*}
\xymatrix{
	\Hom_A^H(\phi_!(M), N) \ar@/^6pt/[r]^-\alpha
		& \Hom_A^G(M, \phi^*(N)) \ar@/^6pt/[l]^-\beta
}
\end{align*}
as follows. Given $f: \phi_!(M) \to N$, for each $g \in G$ and each $m \in M_g$
set $\alpha(f)(m) = f(m) \ot g$. Conversely, given $f: M \to \phi^*(N)$, let 
$\epsilon: \k[G] \to \k$ be the counit of $\k[G]$, i.e. the algebra map 
defined by setting $\epsilon(g) = 1$, and set $\beta(f) = 1 \ot \epsilon 
\circ f$. Direct computation shows that these maps are well defined, natural,
and mutual inverses. Thus $\phi_!$ is the left adjoint of $\phi^*$.

Now we define maps 
\begin{align*}
\xymatrix{
	\Hom_A^G(\phi^*(N), M) \ar@/^6pt/[r]^-\gamma
		& \Hom_A^H(N, \phi_*(M)) \ar@/^6pt/[l]^-\delta
}
\end{align*}
as follows. Given $f: \phi^*(N) \to M$, for each $h \in H$ and each $n \in N_h$
we set $\gamma(f)(n) = (f(n \ot g))_{g \in \phi^{-1}(h)}$. Conversely, given
$f: N \to \phi_*(M)$, for each $g \in G$ and $n \in N_{\phi(g)}$ we have $f(n)
\in \prod_{g' \in \phi^{-1}(h)} M_{g'}$, so we can set $\delta(f)(n \ot g)$ as 
the $g$-th component of $f(n)$. Once again direct computation shows that these 
maps are well defined, natural, and mutual inverses.
\end{proof}

\section{Injective dimension and change of grading}
\label{COG-INJDIM}
Recall that $G,H$ are groups and $\phi: G \to H$ is a group morphism. We set 
$L = \ker \phi$. Throughout this section $A$ denotes a $G$-graded $\k$-algebra.

\paragraph
\label{hom-dim-inequalities}
As stated in the Introduction, a $G$-graded $A$-module is projective if and 
only if it is projective as $A$-module, i.e. the functor $\phi_!$ preserves 
the projective dimension of an object. Our aim is to describe how $\phi_!$ 
affects the injective dimension of an object. We begin by recalling a previous 
result related to this problem.
\begin{Proposition*}[\cite{RZ-twisted}*{Corollaries 3.2.2, 3.2.3}]
Let $M$ be an object of $\Gr_G A$. Then the following hold.
\begin{enumerate}
\item $\projdim^G_A M = \projdim^H_A \phi_!(M)$ and $\injdim^G_A M \leq 
\injdim^H_A \phi_!(M)$. 

\item $\projdim^G_A M \leq \projdim^H_A \phi_*(M)$ and $\injdim^G_A M = 
\injdim^H_A \phi_*(M)$. 
\end{enumerate}
\end{Proposition*}

\paragraph
\label{phi-finite}
The natural inclusion of the direct sum of a family into its 
product gives rise to a natural transformation $\eta: \phi_! \Rightarrow 
\phi_*$. Notice that $\eta(M): \phi_!(M) \to \phi_*(M)$ is an isomorphism if 
and only if for each $h \in H$ the set $\supp M \cap \phi^{-1}(h)$ is finite. 
If this happens we say that $M$ is \emph{$\phi$-finite}. The following theorem
follows immediately from Proposition \ref{hom-dim-inequalities}. 
\begin{Theorem*}
If an object $M$ of $\Gr_G A$ is $\phi$-finite then $\injdim^G_A M = 
\injdim_A^H \phi_!(M)$. 
\end{Theorem*}

\begin{Remark*}
If $|L| < \infty$ then every $G$-graded $A$-module is $\phi$-finite. Also, if 
$A$ is $\phi$-finite then every finitely generated $G$-graded $A$-module is 
$\phi$-finite, so this result applies in many usual situations. For example, 
assume $A$ is $\NN^r$-graded for some $r > 0$, i.e. $A$ is $\ZZ^r$-graded
and $A_\xi = 0$ if $\xi \notin \NN^r$. Let $\psi: \ZZ^r \to \ZZ$ 
be the morphism $\psi(z_1, \ldots, z_r) = z_1 + \cdots + z_r$. Then $\psi_!(A)$
is $\ZZ$-graded, and furthermore $A_z = 0$ if $z \notin \NN$. Since for each 
$z \in \NN$ the set $\psi^{-1}(z) \cap \NN^r$ is finite, the algebra $A$ is 
$\psi$-finite. Applying the theorem we see that $\injdim_A^{\ZZ^r} A = 
\injdim_A^\ZZ \psi_!(A)$. If $A$ is also noetherian then by 
\cite{Lev-ncreg}*{3.3 Lemma} we see that $\injdim_A^{\ZZ^r} A = \injdim_A A$.
\end{Remark*}

\paragraph
The algebra $\k[G]$ is a $G$-graded $\k$-algebra, and hence
through $\phi$ it is also an $H$-graded algebra, so we may consider the 
category of $H$-graded $\k[G]$-modules $\Gr_H \k[G]$. The algebra 
$\k[H]$ is an object in this category with its usual $H$-grading and the 
action of $\k[G]$ induced by $\phi$. By \cite{Mont-hopf-book}*{Theorem 8.5.6}, 
the functor $- \ot \k[H]: \Mod \k[L] \to \Gr_H \k[G]$ is an equivalence of 
categories. In particular the projective dimension of $\k[H]$ in $\Gr_H \k[G]$
equals $\projdim_{\k[L]} \k$.

\paragraph
\label{P:resolution}
Given an object $N$ of $\Gr_H A$ we denote by $\mathcal S(N)$ the smallest 
subclass of objects of $\Gr_H A$ containing the set $\{\phi_!(\phi^*(N[h])) 
\mid h \in H\}$ and closed under direct sums and direct summands.
\begin{Proposition*}
Set $d = \projdim_{\k[G]}^H \k[H] = \projdim_{\k[L]} \k$. Every $H$-graded 
$A$-module $N$ has a resolution of length at most $d$ by objects of 
$\mathcal S(N)$.
\end{Proposition*}
\begin{proof}
We begin by defining a functor $D_N: \Gr_H \k[G] \to \Gr_H A$.
Given an object $V$ of $\Gr_H \k[G]$, the tensor product $N \ot V$ is an
$A$-module with action induced by the map $\rho: A \to A \ot \k[G]$, and
we set $D_N(V)$ to be the $A$-submodule $\bigoplus_{h \in H} N_h \ot V_h$, 
with the obvious $H$-grading. Given a morphism $f: V \to W$ in $\Gr_H \k[G]$, 
we set $D_N(f)$ as the restriction and correstriction of $\Id_N \ot f$.

Fix $h \in H$. By definition $D_N(\k[G][h])$ and $\phi_!(\phi^*(N[h^{-1}]))[h]$
are $A$-submodules of $N \ot \k[G]$, and it is immediate to check that in both 
cases the homogeneous component of degree $h' \in H$ is $N_h \ot \k[G]_{hh'}$, 
so in fact these two $H$-graded $A$-modules are equal. Furthermore, if $P$ is 
any projective object in $\Gr_H \k[G]$ then there exists an object $Q$ such 
that $P \oplus Q$ is a free $H$-graded $\k[G]$-module, which is isomorphic to 
$\bigoplus_{i \in I} (\k[G])[h_i]$ for some index set $I$, not necessarily 
finite, with $h_i \in H$. Now $D_N$ commutes with direct summs, $D_N(P)$ is a 
direct summand of $D_N(P \oplus Q) \cong \bigoplus_{i \in I} D_N(\k[G][h_i]) = 
\bigoplus_{i \in I} \phi_!(\phi^*(N[h_i^{-1}]))[h_i]$, which obviously lies in 
$\mathcal S(N)$. 

For each $h \in H$ we define a map $n \in N_h \mapsto n \ot h \in 
D_N(\k[H])_h$; the direct sum of these maps gives us an isomorphism $N \cong 
D_N(\k[H])$. Taking a projective resolution $P^\bullet$ of $\k[H]$ of 
length $d$ and applying $D_N$, we obtain a complex $D_N(P^\bullet) \to 
D_N(\k[H]) \cong N$; since $\k[G]$ is a free $\k$-module, projective 
$\k[G]$-modules are projective over $\k$ so this is an exact complex, and
from the previous paragraph we see that it is a resolution of $N$ by objects
in $\mathcal S(N)$.
\end{proof}

\paragraph
Let $M$ be a $G$-graded $A$-module. Recall that $\phi^*(\phi_!(M)) \subset M 
\ot \k[G]$ consists of all $m \ot g'$ with $m \in M_{g}$ and $\phi(g) = 
\phi(g')$. For each $l \in L$ we have a map $M[l] \to \phi^*\phi_!(M)$ whose 
homogeneous component of degree $g \in G$ is given by $m \in M[l]_g \mapsto m 
\ot gl \in \phi^*\phi_!(M)$. This induces a natural map $\bigoplus_{l \in L} 
M[l] \to \phi^*\phi_!(M)$. This map has an inverse, given by $m \ot g' \in 
\phi^*(\phi_!(M))\mapsto m \in M[g^{-1}g']$, so we get a natural isomorphism 
$\phi^*(\phi_!(M)) \cong \bigoplus_{l \in L} M[l]$. This observation is used 
in the following lemma.

\label{L:acyclic}
\begin{Lemma*}
Assume $A$ is left $G$-graded noetherian. Let $I, M$ be objects of $\Gr_G A$ 
with $I$ injective, and let $N$ be a direct summand of $\phi_!(M)$. Then 
$\R^i\Hom_A^H(N, \phi_!(I)) = 0$ for all $i > 0$.
\end{Lemma*}
\begin{proof}
It is enough to show that the result holds for $N = \phi_!(M)$. In that case 
we have isomorphisms
\begin{align*}
\Hom_A^H(\phi_!(M), \phi_!(I)) 
 &\cong \Hom_A^G(M, \phi^*(\phi_!(I)))
 \cong \Hom_A^G\left(M, \bigoplus_{l \in L} I[l] \right).
\end{align*}
Since this isomorphism is natural in the first variable, we obtain for each 
$i \geq 0$ an isomorphism
\begin{align*}
\R^i \Hom_A^H(\phi_!(M), \phi_!(I)) 
 &\cong\R^i\Hom_A^G\left(M, \bigoplus_{l \in L} I[l] \right).
\end{align*}
Now by the graded version of the Bass-Papp Theorem (see 
\cite{GW-noetherian-book}*{Theorem 5.23} for a proof in 
the ungraded case, which adapts easily to the graded context), the fact that 
$A$ is left $G$-graded noetherian implies that $\bigoplus_{l \in L} I[l]$ is 
injective, and hence the last isomorphism implies $\R^i \Hom_A^H(\phi_!(M), 
\phi_!(I)) = 0$.
\end{proof}

\begin{Remark*}
We point out that the proof does not use the full Bass-Papp Theorem, just the
fact that the direct sum of an arbitrary family of shifted copies of the same 
injective module is again injective, so we may wonder whether this property is 
weaker than $G$-graded noetherianity. In the ungraded case a module is called 
$\Sigma$-injective if the direct sum of arbitrarily many copies of it is
injective. Say that a $G$-graded $A$-module is graded $\Sigma$-injective if an
arbitrary direct sum of shifted copies of itself is injective. Then by a 
reasoning analogous to that of \cite{FW-direcsumreps}*{Theorem, pp. 205-6}
one can prove that an algebra is left $G$-graded noetherian if and only if
every injective object of $\Gr_G A$ is graded $\Sigma$-injective. We thank
MathOverflow user Fred Rohrer for the reference.
\end{Remark*}

\paragraph
\label{T:main-theorem}
We are now ready to prove the main result of this section.
\begin{Theorem*}
Set $d = \projdim_{\k[L]} \k$. Assume $A$ is left $G$-graded noetherian. For 
every object $M$ of $\Gr_G A$ we have $\injdim_A^G M \leq \injdim^H_A 
\phi_!(M) \leq \injdim^G_A M + d$
\end{Theorem*}
\begin{proof}
The first inequality holds by Proposition \ref{hom-dim-inequalities}. The case 
where $M$ is of infinite injective dimension is trivially true, so let
us consider the case where $n = \injdim_A^G M$ is finite. In this case we work
by induction.

If $n = 0$ then $M$ is injective in $\Gr_G A$. Let $N$ be an object of 
$\Gr_H A$, and let $P^\bullet \to N$ be a resolution of $N$ of length $d$ by
objects of $\mathcal S(N)$ as in Proposition \ref{P:resolution}. It follows 
from Lemma \ref{L:acyclic} that $\R^i \Hom_A^H(P,\phi_!(I)) = 0$ for every 
object $P$ of $\mathcal S(N)$, so in fact $P^\bullet$ is an acyclic resolution
of $N$ and
\begin{align*}
\R^i\Hom_A^H(N, \phi_!(M)) 
 &\cong H^i(\Hom_A^H(P^\bullet, \phi_!(M)))
\end{align*}
for each $i \geq 0$. Thus $\R^i\Hom_A^H(N, \phi_!(M)) = 0$ for all $i > d$, 
and since $N$ was arbitrary this implies that $\injdim_A^H \phi_!(M) \leq d$.

Now assume that the result holds for all objects of $\Gr_G A$ with
injective dimension less than $n$. Let $M \to I$ be an injective envelope
of $M$ in $\Gr_G A$, and let $M'$ be its cokernel. Then $\injdim^G_A M' = 
n-1$, and so by the inductive hypothesis $\injdim^H_A \phi_!(M') \leq n-1+d$. 
Now we have an exact sequence in $\Gr_H A$ of the form
\begin{align*}
0 \to \phi_!(M) \to \phi_!(I) \to \phi_!(M') \to 0.
\end{align*}
By standard homological algebra the injective dimension of $\phi_!(M)$ is 
bounded above by the maximum between $\injdim_A^H \phi_!(I) + 1 \leq d + 1$
and $\injdim_A^H \phi_!(M') + 1 \leq n + d$. This gives us the desired 
inequality.
\end{proof}

\section{Change of grading at the derived level and dualizing complexes}
\label{COG-DC}
Dualizing complexes for noncommutative rings were introduced by A. Yekutieli 
in the context of connected $\NN$-graded algebras in order to study their 
local cohomology; they have proven to be very useful in the study of ring 
theoretical properties of non commutative rings, see for example 
\cites{Yek-dc, Jor-lc, VdB-existence-dc, YZ-aus-dc, WZ-survey-dc, 
YZ-rigid-dc}, etc. A dualizing complex is essentially an object $R^\bullet$ in 
the derived category of $\Mod A^e$ such that the functor $\R\Hom_A(-, 
R^\bullet)$ is a duality between $\D^b(\Mod A)$ and $\D^b(\Mod A^\op)$, for a 
precise definition see Definition \ref{dc-definition}. A graded dualizing 
complex in principle only guarantees dualities at the graded level, but 
according to Van den Bergh, a $\ZZ$-graded dualizing complex is also an 
ungraded dualizing complex \cite{VdB-existence-dc}. In this section we show
that in fact a $\ZZ^r$-graded dualizing complex remains a dualizing complex
after regrading. Once you have Theorem \ref{T:main-theorem}, the proof
in the $\ZZ^r$-graded case is no more difficult than in the $\ZZ$-graded case,
except for the technical complications due to the extra gradings. Still, we 
felt it was worthwhile to develop these technicalities in order to obtain a 
precise statement of Theorem \ref{dc-regrading}.

Throughout this section $\k$ is a field, $G$ is an abelian group, and $A$ is a 
$G$-graded $\k$-algebra. We denote by $A^e$ the enveloping algebra $A \ot 
A^\op$; since $G$ is abelian both $A^\op$ and $A^e$ are $G$-graded algebras. 

\paragraph
Let us fix some notation regarding derived categories. Given an abelian 
category $\A$, we denote by $\K(A)$ the category of complexes of objects of 
$\A$ with homotopy classes of maps of complexes as morphisms, and by $\D(\A)$ 
the derived category of $\A$. As usual we denote by $\D^+(\A), \D^-(\A), 
\D^b(\A)$ the full subcategories of $\D(A)$ consisting of left bounded, right 
bounded and bounded complexes. Recall that an injective resolution of a 
left bounded complex $R^\bullet$ is a quasi-isomorphism $R^\bullet \to 
I^\bullet$ where $I^\bullet$ is a left bounded complex formed by injective 
objects of $\A$. If $\A$ has enough injectives then every left bounded complex 
has an injective resolution. Analogous remarks apply for projective 
resolutions of right bounded complexes.

If $F: \A \to \mathcal B$ is an exact functor between abelian categories, then
by the universal property of derived categories there is an induced functor
$\D(\A) \to \D(\mathcal B)$, which by abuse of notation we will also denote by
$F$.

\paragraph
The maps $a \in A \mapsto a \ot 1 \in A^e$ and $a \in A^\op \mapsto 1 \ot a
\in A^e$ induce restriction functors $\Res_A: \Gr_{G} A^e \to \Gr_{G} A$ and 
$\Res_{A^\op}: \Gr_{G} A^e \to \Gr_{G} A^\op$. These functors are exact and 
preserve projectives and injectives, which can be proved following the lines 
of the proof in the case $G = \ZZ$ found in \cite{Yek-dc}*{Lemma 2.1}. If
$H$ is any group and $\phi: G \to H$ is a group morphism then it is
clear that the associated change of grading functors commute with the 
restriction functors in the obvious sense. Since restriction and change of 
grading functors are exact, they induce exact functors between the 
corresponding derived categories.

\paragraph
There exists a functor
\begin{align*}
\HOM_A^{G}: \K(\Gr_{G} A^e)^\op \times \K(\Gr_{G} A^e) 
 \to \K(\Gr_{G} A^e)
\end{align*}
defined as follows. Given complexes $M^\bullet, N^\bullet$, for each $n \in 
\ZZ$ we set
\begin{align*}
 \HOM_A^{G}(N^\bullet,M^\bullet)^n 
  &= \prod_{p \in \ZZ} \HOM_A^{G}(N^p, M^{p+n}),
\end{align*}
where the product is taken in the category of $G$-graded $A^e$-modules;
this sequence of $G$-graded $A^e$-modules is made into a complex with 
differential
\begin{align*}
 d^n &= \prod_{p \in \ZZ} ((-1)^{n+1}\HOM_A^{G}(d_N^{p},M^{p+n}) + 
 \HOM_A^{G}(N^{p},d_M^{p+n})).
\end{align*}
The action of $\HOM_A^G$ on maps is defined in the usual way.

The functor $\HOM_A^{G}$ has a right derived functor
\begin{align*}
  \R\HOM_A^{G}: 
  \D(\Gr_{G} A^e)^\op \times \D(\Gr_{G} A^e) 
  \to \D(\Gr_{G} A^e).
\end{align*}
When $M^\bullet$ is an object of $\D^+(\Gr_{G} A^e)$ such that 
$M^i$ is injective as left $A$-module for each $i \in \ZZ$, then 
\[
 \R\HOM_A^{G}(N^\bullet, M^\bullet) 
  \cong \HOM_A^{G}(N^\bullet,M^\bullet)
\] 
for every object $N^\bullet$ of $\D(\Gr_{G} A^e)$. Analogously, if
$N^\bullet$ is an object of $\D^-(\Gr_{G} A^e)$ such that $N^i$ is 
projective as left $A$-module for each $i \in \ZZ$, then 
\[
\R\HOM_A^{G}(N^\bullet, M^\bullet) 
 \cong \HOM_A^{G}(N^\bullet, M^\bullet)
\]
for every object $M^\bullet$ of $\D(\Gr_{G} A^e)$. This is proved in the case
$G = \ZZ$ in \cite{Yek-dc}*{Theorem 2.2}, and the general proof follows the 
same reasoning. There is a completely analogous functor $\HOM^{G}_{A^\op}$ 
whose derived functor $\R\HOM_{A^\op}^{G}$ has similar properties.

\paragraph
\label{natural-map}
Let $R^\bullet$ be a complex of $A^e$-modules. Seeing $A^\op$ as a
complex of $A^e$-modules concentrated in homological degree $0$, there is a
map $A^\op \to \HOM_A^{G}(R^\bullet, R^\bullet)$ given by sending $a 
\in A^\op$ to right multiplication by $a$ acting on $R^\bullet$. 
Now let $P^\bullet \to R^\bullet$ be a projective resolution of $R^\bullet$, 
so there is an isomorphism 
\[
\R\HOM_A^{G}(R^\bullet, R^\bullet) 
 \cong \HOM_A^{G}(P^\bullet, P^\bullet),
\]
and we get a map $\nat_A: A^\op \to \R\HOM_A^{\ZZ^r}(R^\bullet, R^\bullet)$. 
This map is independent of the projective resolution we choose, so we refer to 
it as the \emph{natural map} from $A^\op$ to $\R\HOM_A^{\ZZ^r}(R^\bullet, 
R^\bullet)$. In the same way there is a natural map from $A$ to 
$\R\HOM_{A^\op}^{\ZZ^r}(R^\bullet, R^\bullet)$. The proof that these maps are 
independent of the chosen resolution is quite tedious but elementary; the 
reader is referred to \cite{Zad-thesis}*{Appendix A} for details.

\paragraph
\label{dc-definition}
Assume that $G = \ZZ^r$ for some $r \geq 0$. We say that $A$ is 
$\NN^r$-graded if $\supp A \subset \NN^r$, and that it is connected if $A_0 = 
\k$. If $A$ is $\NN^r$-graded then so are $A^\op$ and $A^e$, and they are 
connected if and only if $A$ is connected. 

The following definition is adapted from \cite{Yek-dc}*{Definition 3.3}.
\begin{Definition*}
Let $A$ be a connected $\NN^r$-graded noetherian algebra. A 
\emph{$\ZZ^r$-graded dualizing complex} over $A$ is a bounded complex 
$R^\bullet$ of $A^e$-modules with the following properties.

\begin{enumerate}
\item 
\label{fg-dc}
The cohomology modules of $\Res_A(R^\bullet)$ and $\Res_{A^\op}(R^\bullet)$ 
are finitely generated.

\item 
\label{inj-dc}
Both $\Res_A(R^\bullet)$ and $\Res_{A^\op}(R^\bullet)$ have finite
injective dimension.

\item 
\label{nat-dc}
The maps $\nat_A: A^\op \to \R\HOM_{A}^{\ZZ^r} (R^\bullet, 
R^\bullet)$ and $\nat_{A^\op}: A \to \mathcal R\HOM_{A^\op}^{\ZZ^r}(R^\bullet, 
R^\bullet)$ are isomorphisms in $\D(\Gr_{\ZZ^r} A^e)$.
\end{enumerate}
\end{Definition*}
A dualizing complex in the ungraded sense is an object of $\D(\Mod A^e)$ which
complies with the ungraded analogue of the previous definition. Our objective 
is to show that a $\ZZ^r$-graded dualizing complex remains a dualizing complex 
if we change (or forget) the grading. Since being finitely generated is 
independent of grading, item \ref{fg-dc} of the definition remains true if we 
change or forget the grading. To see how item \ref{inj-dc} behaves with 
respect to change of grading requires a derived version of Theorem 
\ref{T:main-theorem}, while item \ref{nat-dc} is also invariant by change
of grading by a simple argument. We provide the details in the following 
lemmas, in a slightly more general context.

\paragraph
\label{derived-inj-dim}
Recall that given a group morphism $\phi: G \to H$, a $G$-graded $\k$-vector 
space $M$ is said to be $\phi$-finite if $\supp M \cap \phi^{-1}(h)$ is a 
finite set for each $h \in H$. 

\begin{Lemma*}
Let $\phi: G \to H$ be a group morphism and set $L = \ker \phi$. Let 
$R^\bullet$ be a bounded complex of $G$-graded $A$-modules. 
\begin{enumerate}
\item 
\label{phi-finite-derived-injdim}
If the cohomology modules of $R^\bullet$ are $\phi$-finite then
$\injdim_A^{G} R^\bullet = \injdim_A^H \phi_!(R^\bullet)$

\item 
\label{noetherian-derived-injdim}
Let $d = \projdim_{\k[L]} \k$. If $A$ is left $G$-graded noetherian then 
the following inequalities hold
\[
 \injdim_A^G R^\bullet 
  \leq \injdim_A^H \phi_!(R^\bullet) 
  \leq \injdim_A^G R^\bullet + d.
\]
\end{enumerate}
\end{Lemma*}
\begin{proof}
Let $R^\bullet \to I^\bullet$ be an injective resolution of minimal length.
It is enough to prove the statement with $I^\bullet$ instead of $R^\bullet$.

Suppose $I^\bullet$ has $\phi$-finite cohomology modules. Recall that there is 
a natural transformation $\eta: \phi_! \Rightarrow \phi_*$, and that $\eta(M)$ 
is an isomorphism if an only if $M$ is $\phi$-finite. The class of 
$\phi$-finite $G$-graded $A$-modules is closed by extensions, so applying 
\cite{Hart-RD}*{Proposition 7.1} (in the reference ``thick'' stands for 
``closed by extensions'') we get that the map $\phi_!(I^\bullet) \to 
\phi_*(I^\bullet)$ is a quasi-isomorphism, and since $\phi_*$ preserves 
injectives it is an injective resolution, so $\injdim_A^{G} R^\bullet \geq 
\injdim_A^H \phi_!(R^\bullet)$. If the inequality were strict, then we could 
truncate $\phi_*(I^\bullet)$ to obtain a shorter complex of the form
\[
\cdots 
 \to \phi_*(I^{j-1}) 
 \to \phi_*(I^{j}) 
 \to \phi_*(\coker d^j) 
 \to 0 
 \to \cdots
\]
with $\phi_*(\coker d^j)$ an injective $H$-graded $A$-module. Since $\phi_*$
preserves injective dimension by Proposition \ref{hom-dim-inequalities}, this
would contradict the fact that $I^\bullet$ is a minimal resolution of 
$R^\bullet$, so in fact $\injdim_A^{G} R^\bullet = \injdim_A^H 
\phi_!(R^\bullet)$. This proves item \ref{phi-finite-derived-injdim}

For item \ref{noetherian-derived-injdim}, assume first that $I^\bullet$ is 
bounded. We proceed by induction on $s$, the length of $I^\bullet$. The case 
$s = 0$ is a special case of Theorem \ref{T:main-theorem}. Now let $t \in \ZZ$ 
be the minimal homological degree such that $I^t \neq 0$, and consider the 
exact sequence of complexes
\begin{align*}
0 \to I^{> t} \to I^\bullet \to I^t \to 0,
\end{align*}
where $I^t$ is seen as a complex concentrated in homological degree $t$ and
$I^{> t}$ is the subcomplex of $I^\bullet$ formed by all components in 
homological degree larger than $t$. Thus there is a distinguished triangle
$\phi_!(I^{> t}) \to \phi_!(I^\bullet) \to \phi_!(I^t) \to$ in $\D(\Gr_H A)$.
By the inductive hypothesis the inequality holds for the first and third 
complexes of the triangle, so a simple argument with long exact sequences 
shows that the corresponding inequality holds for $\phi_!(I^\bullet)$.

Finally, if $I^\bullet$ is not bounded then we only have to prove that 
$\phi_!(I^\bullet)$ does not have finite injective dimension. Now $\phi^*$
preserves injective dimensions, and since $\phi^*(\phi_!(I^\bullet)) \cong 
\bigoplus_{l \in L} I[l]^\bullet$ has infinite injective dimension, so does
$\phi_!(I^\bullet)$.
\end{proof}

\begin{Lemma}
\label{natural-map-lemma}
Let $G,H$ be abelian groups and $\phi: G \to H$ a group morphism. Assume $A$ 
is $G$-graded noetherian. Let $S^\bullet, R^\bullet$ be bounded 
complexes of $G$-graded $A^e$-modules such that the cohomology modules of 
$R^\bullet$ are finitely generated as left $A$-modules. 
\begin{enumerate}
\item 
\label{RHOM-cog-iso}
The map
\[
 \phi_!(\R\HOM_A^G(R^\bullet, S^\bullet)) 
  \to \R\HOM_A^H(\phi_!(R^\bullet), \phi_!(S^\bullet))
\]
is an isomorphism.

\item 
\label{nat-cog}
The composition 
\begin{align*}
\xymatrix{
\phi_!(A) 
  \ar[r]^-{\phi_!(\nat_A)}
  &\phi_!(\R\HOM_A^G(R^\bullet, R^\bullet)) \ar[r]
  &\R\HOM_A^H(\phi_!(R^\bullet), \phi_!(R^\bullet))
}
\end{align*}
equals $\nat_{\phi_!(A)}: \phi_!(A) \to \R\Hom_A^H(\phi_!(R^\bullet), 
\phi_!(R^\bullet))$
\end{enumerate}
\end{Lemma}
\begin{proof}
The map from item \ref{RHOM-cog-iso} is obtained as follows. Let $P^\bullet
\to R^\bullet$ be a projective resolution. Then $\phi_!(P^\bullet) \to 
\phi_!(R^\bullet)$ is also a projective resolution since $\phi_!$ is exact and 
preserves projectives. Now by definition of $\HOM^G_A(R^\bullet, S^\bullet)$, 
we have $\phi_!(\HOM_A^{G}(P^\bullet, S^\bullet)) \subset \HOM_A^{H}
(\phi_!(P^\bullet), \phi_!(S^\bullet))$, and the desired map is the inclusion. 
Once again this map is independent of the chosen projective resolution. 
Clearly item \ref{nat-cog} follows from this.

If $R^\bullet$ and $S^\bullet$ are concentrated in homological degree $0$,
item \ref{RHOM-cog-iso} is a well-known result, see for example 
\cite{RZ-twisted}*{Proposition 1.3.7}. The general result follows by standard 
arguments using \cite{Hart-RD}*{Proposition I.7.1(i)}.
\end{proof}

\paragraph
\label{dc-regrading}
We are now ready to prove the main result of this section. 
\begin{Theorem*}
Let $A$ be a connected $\NN^r$-graded noetherian $\k$-algebra and let
$R^\bullet$ be a $\ZZ^r$-graded dualizing complex over $A$.
\begin{enumerate}
\item 
\label{phi-connected-dc}
Let $s > 0$ and let $\phi: \ZZ^r \to \ZZ^s$ be a group
morphism such that $\phi_!(A)$ is $\NN^s$-graded connected. Then
$\phi_!(R^\bullet)$ is a $\ZZ^s$-graded dualizing complex over 
$\phi_!(A)$ of injective dimension $\injdim^{\ZZ^r}_A R^\bullet$.

\item 
\label{ungrading-dc}
Let $\O: \D(\Gr_{\ZZ^r} A^e) \to \D(\Mod A^e)$ be the
forgetful functor. Then $\O(R^\bullet)$ is a dualizing complex over $A$
in the ungraded sense, of injective dimension at most $\injdim^{\ZZ^r}_A
R^\bullet + 1$.
\end{enumerate}
\end{Theorem*}
\begin{proof}
Let us prove item \ref{phi-connected-dc}. As we have already noticed, $\phi_!$
commutes with the restriction functors and does not change the fact that a
bimodule is finitely generated as left or right $A$-module, so 
$\phi_!(R^\bullet)$ complies with item \ref{fg-dc} of Definition 
\ref{dc-definition}. Since $A$ is $\ZZ^r$-graded noetherian it is also 
$\ZZ^s$-graded noetherian, and hence $\phi_!(A)$ is locally finite; this 
implies that $A$ is $\phi$-finite, otherwise $\phi_!(A)$ would have a 
homogeneous component of infinite dimension. Since the cohomology modules of 
$R^\bullet$ are finitely generated, they are also $\phi$-finite and hence by 
item \ref{phi-finite-derived-injdim} of Lemma \ref{derived-inj-dim}
$\injdim_A^{\ZZ^s} \phi_!(R^\bullet) = \injdim_A^{\ZZ^r} R^\bullet$, so item
\ref{inj-dc} of Definition \ref{dc-definition} also holds for $R^\bullet$. 
Finally item \ref{nat-dc} of the definition follows immediately from item 
\ref{nat-cog} of Lemma \ref{natural-map-lemma}.

We now prove item \ref{ungrading-dc}. Let $\psi: \ZZ^r \to \ZZ$ be the map 
$\psi(z_1, \ldots, z_r) = z_1 + \cdots + z_r$. Then $A$ is $\psi$-finite and 
$\psi_!(A)$ is connected $\NN$-graded, so by the first item $\psi_!(R^\bullet)$
is a $\ZZ$-graded dualizing complex over $A$ of injective dimension 
$\injdim_A^{\ZZ^r} R^\bullet$. Now a similar reasoning as the one we used for 
the first item, but this time using item \ref{noetherian-derived-injdim} of 
Lemma \ref{derived-inj-dim}, shows that $\O(\psi_!(R^\bullet)) = 
\O(R^\bullet)$ is a dualizing complex and gives the bound for its injective 
dimension.
\end{proof}

\vspace{5em}

\noindent A.S.: \\
IMAS-CONICET y Departamento de Matem\'atica\\
Facultad de Ciencias Exactas y Naturales,\\
Universidad de Buenos Aires,\\
Ciudad Universitaria, Pabell\'on 1\\
1428, Buenos Aires, Argentina.\\
\texttt{asolotar@dm.uba.ar}

\bigskip

\noindent P.Z. :\\
Instituto de Matem\'atica e Estat\'istica, \\
Universidade de S\~ao Paulo. \\
Rua do Matão, 1010 \\
CEP 05508-090 - S\~ao Paulo - SP\\
\texttt{pzadun@ime.usp.br}

\end{document}